\documentclass[11pt,reqno]{amsart}
\usepackage[margin=2.9cm]{geometry}
\usepackage[T1]{fontenc}
\usepackage{amsmath,amssymb,amsthm}
\usepackage{xcolor}
\usepackage{hyperref}
\hypersetup{colorlinks=true,linkcolor=blue!50!black,citecolor=blue!50!black,urlcolor=blue!50!black,
  pdfauthor={Hikmet Burak \"Ozcan}}

\newtheorem{theorem}{Theorem}[section]
\newtheorem{lemma}[theorem]{Lemma}
\theoremstyle{definition}
\newtheorem*{problemstar}{Problem}
\theoremstyle{remark}
\newtheorem{remark}[theorem]{Remark}

\numberwithin{equation}{section}

\DeclareMathOperator{\spec}{sp}
\newcommand{\Z}{\mathbb{Z}}
\newcommand{\Q}{\mathbb{Q}}
\newcommand{\N}{\mathbb{N}}

\title[The spectrum of $(\xi\alpha^n)$ can be uncountable]{The spectrum of $(\xi\alpha^n)$ can be uncountable}

\author{H\.IKMET BURAK \"OZCAN}
\address{Department of Mathematics, \.Izmir Institute of Technology, \.Izmir, Turkey}

\subjclass[2020]{11K06 (Primary), 11K16 (Secondary)}
\keywords{Uniform distribution modulo one, spectrum, exponential sequences, Diophantine approximation}

\date{}

\begin{document}

\begin{abstract}
In this note, we give a counterexample to the assertion of Problem~10.4 in Bugeaud's monograph \emph{Distribution modulo one and Diophantine approximation}, which goes back to Mend\`es France. The problem states that the spectrum of the sequence $(\xi\alpha^n)_{n\ge1}$, that is, the set of irrational $\theta\in(0,1)$ for which $(\xi\alpha^n-n\theta)_{n\ge1}$ is not uniformly distributed modulo one, is at most countable for all real $\xi\ne0$ and $\alpha>1$. More precisely, we prove that for every real $\alpha>1$ there are $2^{\aleph_0}$ real numbers $\xi>0$ for which the spectrum of $(\xi\alpha^n)_{n\ge1}$ contains one and the same uncountable set.
\end{abstract}

\maketitle

\section{Introduction}

A sequence $(x_n)_{n\ge1}$ of real numbers is \emph{uniformly distributed modulo one} (u.d.\ mod~1) if
\begin{equation}\label{eq:ud}
\frac1N\,\#\bigl\{1\le n\le N:\{x_n\}\in[a,b)\bigr\}\longrightarrow b-a\qquad(N\to\infty)
\end{equation}
for all $0\le a<b\le1$, where $\{t\}$ is the fractional part of $t$. We also write $\|t\|$ for the distance from $t$ to the nearest integer. The \emph{spectrum} of a sequence $(x_n)_{n\ge1}$ was introduced by Mend\`es France \cite{MFa}. Following Bugeaud \cite{Bu}, we define it as
\[
\spec\bigl((x_n)_{n\ge1}\bigr)=\bigl\{\theta\in(0,1)\setminus\Q:\ (x_n-n\theta)_{n\ge1}\ \text{is not u.d.\ mod }1\bigr\}.
\]

The following problem was posed by Mend\`es France in \cite{MF} and recorded by Bugeaud in his monograph. It was recently recalled as open by Chen, Ye and Zheng \cite{CYZ}.

\begin{problemstar}[{\cite[Problem~10.4]{Bu}}]
Let $\xi$ be a non-zero real number and $\alpha>1$ be a real number. The spectrum of the sequence $(\xi\alpha^n)_{n\ge1}$ is at most countable.
\end{problemstar}

The assertion looks plausible. In fact, in the following elementary cases, the spectrum is empty. If $\xi$ and $\alpha$ are integers, then $\xi\alpha^n-n\theta\equiv-n\theta\pmod1$, and $(-n\theta)_{n\ge1}$ is u.d.\ mod~1 because $\theta$ is irrational. If $\xi\in\Q$ and $\alpha\in\Z$, then $\alpha^n$ is eventually periodic modulo the denominator of $\xi$, so that $\{\xi\alpha^n\}$ is eventually periodic, say with period $M$. Hence, on each residue class $n\equiv r\pmod M$ the sequence $(\xi\alpha^n-n\theta)_{n\ge1}$ eventually differs from $(-n\theta)_{n\ge1}$ by a constant modulo one, and so it is u.d.\ mod~1 along each of these $M$ arithmetic progressions because $M\theta$ is irrational. Consequently, it is u.d.\ mod~1. If $\alpha$ is a Pisot number and $\xi=1$, then $\|\alpha^n\|\to0$. Let $a_n$ be the nearest integer to $\alpha^n$, so that $\alpha^n-a_n\to0$. The sequence $(a_n-n\theta)_{n\ge1}$ has the same fractional parts as $(-n\theta)_{n\ge1}$ and is therefore u.d.\ mod~1, and hence so is $(\alpha^n-n\theta)_{n\ge1}$ by \cite[Chapter~1, Theorem~1.2]{KN}. Moreover, there are two general facts which support the assertion.

First, the spectrum of an arbitrary real sequence $(x_n)_{n\ge1}$ is a Lebesgue null set. This was already observed by Mend\`es France \cite{MFa}, and we recall the argument. Fix an integer $h\ne0$. Since the functions $\theta\mapsto e^{-2\pi ihn\theta}$ with $1\le n\le N$ are orthonormal on $[0,1]$, we have
\[
\int_0^1\Bigl|\frac1N\sum_{n=1}^Ne^{2\pi ih(x_n-n\theta)}\Bigr|^2\,d\theta=\frac1N .
\]
As this holds for every $h\ne0$ and $\sum_{N\ge1}N^{-1}\cdot N^{-1}<\infty$, the hypothesis of the criterion of Davenport, Erd\H os and LeVeque \cite{DEL} (see also \cite{KN}) is satisfied, and the criterion implies that $(x_n-n\theta)_{n\ge1}$ is u.d.\ mod~1 for almost every $\theta$.

Second, for our sequences the null set can be taken in the variable $\xi$. By van der Corput's difference theorem \cite{KN}, it is enough to show that
\[
\bigl(\xi\alpha^{n+h}-\xi\alpha^n-h\theta\bigr)_{n\ge1}=\bigl(\xi(\alpha^h-1)\alpha^n-h\theta\bigr)_{n\ge1}
\]
is u.d.\ mod~1 for every $h\ge1$. Since uniform distribution does not change when we add the constant $-h\theta$, the exceptional set of $\xi$ does not depend on $\theta$. Moreover, for a fixed $h$, this set is the preimage of $\{\zeta:(\zeta\alpha^n)_{n\ge1}\ \text{is not u.d.\ mod }1\}$ under the dilation $\xi\mapsto\xi(\alpha^h-1)$. Since $(\zeta\alpha^n)_{n\ge1}$ is u.d.\ mod~1 for almost every $\zeta$ by \cite[Corollary~1.9]{Bu}, and since the dilation preserves null sets as $\alpha^h\ne1$, this set is null. As a countable union of null sets is null, we obtain $\spec((\xi\alpha^n)_{n\ge1})=\emptyset$ for almost every $\xi$, which shows that a counterexample has to be exceptional in two ways, once in $\xi$ and once in $\theta$. On the other hand, Mend\`es France also notes in \cite[Probl\`eme~5]{MF} that the spectrum of a general sequence can be uncountable, and the question is whether this happens for the sequences $(\xi\alpha^n)_{n\ge1}$. It does, and the assertion is false.

\begin{theorem}\label{thm:main}
Let $\alpha>1$ be a real number. Then, there exist an uncountable set $C$ of irrational numbers in $(0,1)$ and a set $\Xi_\alpha\subset(0,\infty)$ of cardinality $2^{\aleph_0}$ such that
\[
C\ \subseteq\ \spec\bigl((\xi\alpha^n)_{n\ge1}\bigr)\qquad\text{for every }\xi\in\Xi_\alpha .
\]
In particular, the spectrum of $(\xi\alpha^n)_{n\ge1}$ is uncountable for every $\xi\in\Xi_\alpha$, and so Problem~\textup{10.4} has a negative answer.
\end{theorem}

In \cite{CYZ}, Chen, Ye and Zheng study two variants of the spectrum, in which $(x_n-n\theta)_{n\ge1}$ is required to have only finitely, respectively only countably, many limit values modulo one. For algebraic $\alpha$, they show that both variants of the spectrum of $(\xi\alpha^n)_{n\ge1}$ are contained in $\Q/\Z$, and in particular that they are countable. Thus, these variants behave as Problem~\textup{10.4} predicts, whereas by Theorem~\ref{thm:main} the spectrum itself does not, for any $\alpha>1$.

When $\alpha$ is an integer, one can prescribe the digits of $\xi$, and then a single $\theta$ in the spectrum is easy to find. Our base $\alpha$ need not be an integer, and Lemma~\ref{lem:shadow} replaces the digit expansion. The main difficulty is a different one, since one and the same $\xi$ has to resonate with uncountably many $\theta$ at the same time. We construct the tree of Section~\ref{sec:tree} for this purpose. Moreover, the set $C$ of Theorem~\ref{thm:main} depends on $\alpha$ only through an integer $q$, and we make this precise in Remark~\ref{rem:uniform}.

Our proof of the failure of uniform distribution is direct, and we work with the definition \eqref{eq:ud} rather than with the exponential sums of Weyl's criterion. Lemma~\ref{lem:shadow} controls $\xi\beta^n$ modulo one up to an error of $1/(2(\beta-1))$, and this error has to be small. We therefore apply the lemma to the base $\beta=\alpha^q$ for a suitable integer $q\ge1$, which is $q=1$ when $\alpha>3$ (see Remark~\ref{rem:range}), and we read off the result along the indices divisible by $q$. For infinitely many $N$, a fixed proportion of the points $\{\xi\alpha^m-m\theta\}$ with $1\le m\le N$ lies in a single short arc, centred either at $0$ or at $\tfrac12$, and this proportion exceeds the length of the arc, so that \eqref{eq:ud} fails.

\section{A shadowing lemma}

The main goal of this section is to prove a shadowing lemma for a real base. In an integer base $b$, one can prescribe the digits of $\xi$ so that $\xi b^n\bmod1$ follows any given target sequence up to an error of $1/b$, and this is elementary. However, when $\alpha$ is not an integer, the digits of $\xi$ no longer control $\xi\alpha^n\bmod1$, since $\alpha$ times an integer need not be an integer. Instead, we prescribe an integer approximation to $\xi\alpha^n$ directly, which is the method of Tijdeman in his work on Mahler's $\tfrac32$-problem \cite{Ti} (see also \cite[Theorem~2.14]{Bu}). We round $\alpha(x_n+y_n)-y_{n+1}$ to the nearest integer at each step, which gives the two-sided bound in \eqref{eq:shadow}.

\begin{lemma}\label{lem:shadow}
Let $\alpha>2$ be a real number, let $m\ge1$ be an integer and let $(y_n)_{n\ge0}$ be any sequence in $[0,1)$. We define integers $x_n$ by
\[
x_0=m,\qquad x_{n+1}=\Bigl\lfloor\alpha(x_n+y_n)-y_{n+1}+\tfrac12\Bigr\rfloor\quad(n\ge0).
\]
Then, the limit $\xi:=\lim_{n\to\infty}x_n\alpha^{-n}$ exists and satisfies
\begin{equation}\label{eq:shadow}
\bigl\|\xi\alpha^n-y_n\bigr\|\ \le\ \frac1{2(\alpha-1)}\qquad\text{for every }n\ge0,
\end{equation}
and moreover $\bigl|\xi-(m+y_0)\bigr|\le\tfrac1{2(\alpha-1)}$. In particular, $\xi>m-\tfrac1{2(\alpha-1)}>0$.
\end{lemma}

\begin{proof}
Put $z_n:=x_n+y_n$ and $e_{n+1}:=z_{n+1}-\alpha z_n$. As $\lfloor t\rfloor$ is the unique integer in $(t-1,t]$, we have $\lfloor t+\tfrac12\rfloor-t\in(-\tfrac12,\tfrac12]$ for every real number $t$. Applying this with $t=\alpha z_n-y_{n+1}$, for which $x_{n+1}=\lfloor t+\tfrac12\rfloor$, we get
\begin{equation}\label{eq:delta}
e_{n+1}=x_{n+1}+y_{n+1}-\alpha z_n=\lfloor t+\tfrac12\rfloor-t\in\bigl(-\tfrac12,\tfrac12\bigr] .
\end{equation}
Dividing $z_{n+1}=\alpha z_n+e_{n+1}$ by $\alpha^{n+1}$ yields $z_{n+1}\alpha^{-(n+1)}=z_n\alpha^{-n}+e_{n+1}\alpha^{-(n+1)}$, and summing this telescopically over $0\le n<k$ we obtain
\begin{equation}\label{eq:xk}
z_k\alpha^{-k}=z_0+\sum_{j=1}^k e_j\alpha^{-j}\qquad(k\ge0).
\end{equation}
By \eqref{eq:delta} we have $|e_j\alpha^{-j}|\le\tfrac12\alpha^{-j}$, and hence the series $\sum_je_j\alpha^{-j}$ converges absolutely. Thus, $\xi:=z_0+\sum_{j\ge1}e_j\alpha^{-j}$ is well defined, and $z_k\alpha^{-k}\to\xi$ by \eqref{eq:xk}. Since $0\le y_k<1$, we also have $x_k\alpha^{-k}=z_k\alpha^{-k}-y_k\alpha^{-k}\to\xi$. Multiplying $\xi$ by $\alpha^k$ and isolating the first $k$ terms using \eqref{eq:xk}, we get
\[
\xi\alpha^k=\alpha^k\Bigl(z_0+\sum_{j=1}^ke_j\alpha^{-j}\Bigr)+\sum_{j>k}e_j\alpha^{k-j}=z_k+T_k,\qquad T_k:=\sum_{i\ge1}e_{k+i}\alpha^{-i}.
\]
By \eqref{eq:delta} and $\sum_{i\ge1}\alpha^{-i}=\tfrac1{\alpha-1}$, we have
\begin{equation}\label{eq:Tk}
|T_k|\ \le\ \frac12\sum_{i\ge1}\alpha^{-i}\ =\ \frac1{2(\alpha-1)}\qquad(k\ge0).
\end{equation}
Since $x_k\in\Z$, the numbers $\xi\alpha^k-y_k$ and $T_k=\xi\alpha^k-x_k-y_k$ differ by an integer, and therefore
\[
\|\xi\alpha^k-y_k\|=\|T_k\|\le|T_k|\le\frac1{2(\alpha-1)},
\]
which proves \eqref{eq:shadow}. Taking $k=0$, we get $\xi=z_0+T_0=m+y_0+T_0$, and \eqref{eq:Tk} yields $|\xi-(m+y_0)|\le\tfrac1{2(\alpha-1)}$. Moreover, every $e_j$ is strictly larger than $-\tfrac12$ by \eqref{eq:delta}, so that $T_0>-\tfrac1{2(\alpha-1)}$ and $\xi>m+y_0-\tfrac1{2(\alpha-1)}\ge m-\tfrac1{2(\alpha-1)}$. Finally, $\alpha>2$ implies $\tfrac1{2(\alpha-1)}<\tfrac12<1\le m$, and hence $m-\tfrac1{2(\alpha-1)}>0$.
\end{proof}

\section{Resonant frequencies}\label{sec:tree}

In this section, we construct a set $C$ from which the set of Theorem~\ref{thm:main} is obtained in Section~\ref{sec:proof}. Our strategy is to build a Cantor-type set from nested intervals $J_j$ with centres $\theta_j$. We partition the indices $n\ge1$ into blocks $B_j$ of rapidly increasing length, and we choose the radius of $J_j$ inversely proportional to the length of $B_j$, so that $n(\theta-\theta_j)$ stays small for all $n\in B_j$ and all $\theta\in J_j$.

Throughout, we fix an integer $R\ge3$ and a real number $\kappa$ with $0<\kappa\le\tfrac14$. These parameters govern the growth of the blocks and the size of the radii, and they will be chosen in \eqref{eq:param}, in terms of the integer $q$ of \eqref{eq:q}. We set $L_j=R^j$ for $j\ge1$, we put $N_0=0$ and $N_j=L_1+\cdots+L_j$, and we let
\[
B_j=\{n\in\Z:N_{j-1}<n\le N_j\}.
\]
Then, the blocks $B_1,B_2,\dots$ partition $\{1,2,3,\dots\}$ and $\#B_j=L_j$. Summing the geometric series, we get $(R-1)N_j=R^{j+1}-R<RL_j$, and therefore
\begin{equation}\label{eq:blocks}
N_j\le\frac{R}{R-1}\,L_j\qquad(j\ge1).
\end{equation}

We index the nodes of the infinite rooted binary tree by $1,2,3,\dots$, where the children of $j$ are $2j$ and $2j+1$, and we assign to the node $j$ the radius $r_j:=\kappa/L_j$. Then, \eqref{eq:blocks} shows that the drift within a block is bounded independently of $j$, namely
\begin{equation}\label{eq:drift}
N_jr_j\ \le\ \frac{\kappa R}{R-1}\qquad(j\ge1).
\end{equation}
Finally, we set the centre of the root to $\theta_1:=\tfrac12$, and we obtain the centres of the two children of a node by shifting its centre to the left and to the right by half of its radius. That is, we define
\[
\theta_{2j}:=\theta_j-\tfrac{r_j}2,\qquad \theta_{2j+1}:=\theta_j+\tfrac{r_j}2,\qquad J_j:=[\theta_j-r_j,\theta_j+r_j].
\]

\begin{lemma}\label{lem:tree}
For every $j\ge1$, the intervals $J_{2j}$ and $J_{2j+1}$ are disjoint subintervals of $J_j$. Consequently, we have $J_j\subseteq J_1\subset(0,1)$ for every $j\ge1$.
\end{lemma}

\begin{proof}
Since $2j,2j+1\ge j+1$, we have $L_{2j},L_{2j+1}\ge L_{j+1}=RL_j$, and thus $r_{2j},r_{2j+1}\le r_j/R\le r_j/3$. The interval $J_{2j}$ of the left child is centred at $\theta_j-\tfrac{r_j}2$ and has radius $r_{2j}<\tfrac{r_j}2$, and hence
\[
J_{2j}=\Bigl[\theta_j-\frac{r_j}2-r_{2j},\ \theta_j-\frac{r_j}2+r_{2j}\Bigr]\subset(\theta_j-r_j,\ \theta_j).
\]
In the same way, we obtain $J_{2j+1}\subset(\theta_j,\theta_j+r_j)$. Thus, one of these two intervals lies in the left half of $J_j$ and the other one in the right half, and so $J_{2j}$ and $J_{2j+1}$ are disjoint subsets of $J_j$. The last statement follows by induction along the chain of ancestors of $j$. Finally, $\kappa\le\tfrac14$ and $R\ge3$ give $r_1=\kappa/R\le\tfrac1{12}$, so that $J_1\subseteq[\tfrac12-\tfrac1{12},\tfrac12+\tfrac1{12}]\subset(0,1)$.
\end{proof}

\begin{lemma}\label{lem:cantor}
There is an uncountable set $C\subset(0,1)$ such that all its elements are irrational and every $\theta\in C$ lies in $J_j$ for infinitely many $j$.
\end{lemma}

\begin{proof}
Let $\omega=(\omega_k)_{k\ge0}\in\{0,1\}^{\N}$. We define the nodes along a branch by setting $n_0(\omega)=1$ and by moving to the left or to the right child at each step, that is,
\[
n_{k+1}(\omega)=2n_k(\omega)+\omega_k .
\]
We also put $\sigma_k=+1$ if $\omega_k=1$, and $\sigma_k=-1$ if $\omega_k=0$. By the recursive definition of the centres, the centre of the $K$-th interval along this branch is the partial sum
\begin{equation}\label{eq:partial}
\theta_{n_K(\omega)}=\tfrac12+\sum_{k<K}\sigma_k\,\frac{r_{n_k(\omega)}}2\qquad(K\ge0).
\end{equation}
Since $n_k\ge1$ and $n_0=1$, we get $n_{k+1}\ge2n_k\ge n_k+1$, and hence by induction
\[
n_k\ge k+1\qquad\text{and}\qquad n_{K+i}\ge n_K+i .
\]
Recalling that $r_j=\kappa/R^j$, we deduce
\begin{equation}\label{eq:radii}
r_{n_k}\le\frac{\kappa}{R^{\,k+1}}\qquad\text{and}\qquad r_{n_{K+i}}\le\frac{\kappa}{R^{\,n_K+i}}=\frac{r_{n_K}}{R^{\,i}}.
\end{equation}
By the first bound in \eqref{eq:radii}, the series
\[
\theta(\omega):=\tfrac12+\sum_{k\ge0}\sigma_k\,\frac{r_{n_k(\omega)}}2
\]
converges absolutely. To see that $\theta(\omega)$ lies in every interval along its branch, we subtract \eqref{eq:partial} from this series, which leaves the tail $\theta(\omega)-\theta_{n_K(\omega)}=\sum_{i\ge0}\sigma_{K+i}\,r_{n_{K+i}}/2$. As $|\sigma_k|=1$, the second bound in \eqref{eq:radii} and $R/(R-1)\le\tfrac32$, which holds since $R\ge3$, give
\[
\bigl|\theta(\omega)-\theta_{n_K(\omega)}\bigr|\le\sum_{i\ge0}\frac{r_{n_{K+i}}}2\le\frac{r_{n_K}}2\sum_{i\ge0}\frac1{R^{\,i}}=\frac{r_{n_K}}2\cdot\frac{R}{R-1}\le\frac34\,r_{n_K}<r_{n_K}.
\]
Therefore, $\theta(\omega)\in J_{n_K(\omega)}$ for every $K\ge0$, and since the indices $n_K(\omega)$ are strictly increasing, $\theta(\omega)$ lies in infinitely many $J_j$.

Next, we show that the map $\omega\mapsto\theta(\omega)$ is injective. Let $\omega\ne\omega'$, and let $\ell$ be the first index where they differ. The two branches share the same nodes up to $\ell$, so that $n_k(\omega)=n_k(\omega')$ for $k\le\ell$, whereas $n_{\ell+1}(\omega)$ and $n_{\ell+1}(\omega')$ are the two distinct children $2n_\ell$ and $2n_\ell+1$. By Lemma~\ref{lem:tree}, the corresponding intervals are disjoint, and they contain $\theta(\omega)$ and $\theta(\omega')$ respectively. Hence, $\theta(\omega)\ne\theta(\omega')$, and the range of $\theta$ has the cardinality $2^{\aleph_0}$ of $\{0,1\}^{\N}$. Now, we set
\[
C:=\bigl\{\theta(\omega):\omega\in\{0,1\}^{\N}\bigr\}\setminus\Q.
\]
Since $\Q$ is countable, $C$ is uncountable. All of its elements are irrational, they lie in $J_1\subset(0,1)$ by Lemma~\ref{lem:tree}, and each of them lies in infinitely many $J_j$ by the first part of the proof.
\end{proof}

\section{Proof of Theorem \ref{thm:main}}\label{sec:proof}

Let $\alpha>1$ be a real number. We choose an integer $q\ge1$ with
\begin{equation}\label{eq:q}
\beta:=\alpha^q>2q+1,
\end{equation}
which is possible because $\alpha^q/q\to\infty$ as $q\to\infty$, and we put $E:=1/(2(\beta-1))$. By \eqref{eq:q}, we have $\beta>3$ and $2(\beta-1)>4q$, that is,
\begin{equation}\label{eq:E}
E<\frac1{4q}\le\frac14 .
\end{equation}
Now, we choose
\begin{equation}\label{eq:param}
R:=4,\qquad \kappa:=\frac1{16q}.
\end{equation}
These are the parameters of Section~\ref{sec:tree}, and they are admissible there, since $R\ge3$ and $\kappa\le\tfrac1{16}\le\tfrac14$. We also write
\[
D:=\frac{\kappa R}{R-1}=\frac1{12q},\qquad \rho:=E+D .
\]
Here, $D$ is the drift bound of \eqref{eq:drift}, and $\rho$ will be the radius of the arcs below. By \eqref{eq:E}, we have
\begin{equation}\label{eq:budget}
\rho\ <\ \frac1{4q}+\frac1{12q}\ =\ \frac1{3q}.
\end{equation}
Finally, let $C$ be the set of Lemma~\ref{lem:cantor} for these $R$ and $\kappa$.

Put $s_0:=0$, $s_1:=\tfrac12$ and $P:=\{0,1\}^{\N}$. For a parameter $c=(c_i)_{i\ge0}\in P$, we define a target sequence by
\[
y^c_0:=0,\qquad y^c_n:=\bigl\{n\theta_j+s_{c_{j-1}}\bigr\}\in[0,1)\quad\text{for }n\in B_j\ (j\ge1).
\]
The value $y^c_0$ is the same for every $c$ and plays no further role. We only specify it because Lemma~\ref{lem:shadow} needs a sequence indexed by $n\ge0$. Applying Lemma~\ref{lem:shadow} to the base $\beta>2$ with $m=1$, we obtain a real number $\xi_c>0$ such that
\begin{equation}\label{eq:close}
\bigl\|\xi_c\beta^n-y^c_n\bigr\|\le E\qquad(n\ge1).
\end{equation}

Now, let $j\ge1$, let $\theta\in J_j$ and let $n\in B_j$. We write $s:=s_{c_{j-1}}$. Since $y^c_n\equiv n\theta_j+s\pmod1$, we have
\[
y^c_n-n\theta-s\ \equiv\ n(\theta_j-\theta)\pmod1 .
\]
As $\theta\in J_j$ and $n\le N_j$, we have $\bigl|n(\theta_j-\theta)\bigr|\le N_jr_j\le D$ by \eqref{eq:drift}. In particular $\bigl\|y^c_n-n\theta-s\bigr\|\le D$, and combining this with \eqref{eq:close} through
\[
\bigl\|\xi_c\beta^n-n\theta-s\bigr\|\le\bigl\|\xi_c\beta^n-y^c_n\bigr\|+\bigl\|y^c_n-n\theta-s\bigr\|
\]
we obtain
\begin{equation}\label{eq:conc}
\bigl\|\xi_c\beta^n-n\theta-s\bigr\|\ \le\ E+D\ =\ \rho .
\end{equation}
Thus, every one of the $L_j$ points $\xi_c\beta^n-n\theta$ of the block $B_j$ lies within $\rho$ of $s$ on the circle.

Next, we pass from the base $\beta$ to the base $\alpha$. Let $\theta\in C$, put $\theta':=\theta/q$, and put
\[
u_m:=\xi_c\alpha^m-m\theta'\qquad(m\ge1).
\]
Then, $\theta'$ is irrational and lies in $(0,1)$, and since $\alpha^{qn}=\beta^n$ and $qn\theta'=n\theta$, we have $u_{qn}=\xi_c\beta^n-n\theta$ for every $n\ge1$. Let $j$ be an index with $\theta\in J_j$, of which there are infinitely many by Lemma~\ref{lem:cantor}, and let $s:=s_{c_{j-1}}$. The $L_j$ indices $m=qn$ with $n\in B_j$ are distinct, they satisfy $m\le qN_j$, and $\|u_m-s\|\le\rho$ for each of them by \eqref{eq:conc}. Hence, by \eqref{eq:blocks} and \eqref{eq:param},
\begin{equation}\label{eq:count}
\#\bigl\{1\le m\le qN_j:\ \|u_m-s\|\le\rho\bigr\}\ \ge\ L_j\ \ge\ \Bigl(1-\frac1R\Bigr)N_j\ =\ \frac3{4q}\cdot qN_j .
\end{equation}
As $c_{j-1}\in\{0,1\}$, one of the two values $s_0$ and $s_1$ occurs as $s$ for infinitely many of these $j$, and we fix such a value $s$. Put $\rho':=1/(3q)$, so that $\rho<\rho'$ by \eqref{eq:budget}. The closed arc $\{t:\|t-s\|\le\rho\}$ is contained in the half-open arc $[s-\rho',s+\rho')$ taken modulo one, which is the union of at most two disjoint intervals of the form $[a,b)$ with $0\le a<b\le1$, of total length $2\rho'$, since $\rho'<\tfrac12$. If $(u_m)_{m\ge1}$ were u.d.\ mod~1, then, applying \eqref{eq:ud} to each of these intervals and adding, the proportion of the indices $1\le m\le N$ with $\{u_m\}$ in this arc would tend to $2\rho'=\tfrac2{3q}$ as $N\to\infty$, whereas by \eqref{eq:count} it is at least $\tfrac3{4q}$ for the infinitely many $N=qN_j$ above. Hence, $(u_m)_{m\ge1}=(\xi_c\alpha^m-m\theta')_{m\ge1}$ is not u.d.\ mod~1, that is, $\theta'\in\spec((\xi_c\alpha^m)_{m\ge1})$. Since $\theta\mapsto\theta/q$ is injective, the set
\[
C':=\bigl\{\theta/q:\theta\in C\bigr\}
\]
is uncountable, all of its elements are irrational numbers in $(0,1)$, and we have shown that
\begin{equation}\label{eq:incl}
C'\subseteq\spec\bigl((\xi_c\alpha^n)_{n\ge1}\bigr)\qquad\text{for every }c\in P .
\end{equation}

It remains to show that distinct parameters give distinct $\xi$. Suppose that $c\ne c'$, and let $i$ be an index with $c_i\ne c'_i$. Put $j=i+1$, so that the block $B_j$ uses the coordinate $c_{j-1}=c_i$, and take any $n\in B_j$, which exists because $\#B_j=L_j\ge4$. Since $y^c_n\equiv n\theta_j+s_{c_i}$ and $y^{c'}_n\equiv n\theta_j+s_{c'_i}$ modulo one, the difference $y^c_n-y^{c'}_n$ is congruent to $\pm(s_1-s_0)=\pm\tfrac12$ modulo one, and so
\[
\bigl\|y^c_n-y^{c'}_n\bigr\|=\tfrac12 .
\]
Suppose that $\xi_c=\xi_{c'}$. Then, \eqref{eq:close}, the subadditivity of $\|\cdot\|$ and \eqref{eq:E} give
\[
\tfrac12=\bigl\|y^c_n-y^{c'}_n\bigr\|\le\bigl\|y^c_n-\xi_c\beta^n\bigr\|+\bigl\|\xi_{c'}\beta^n-y^{c'}_n\bigr\|\le2E<\tfrac12 ,
\]
which is false. Hence, the map $c\mapsto\xi_c$ is injective, and the set
\[
\Xi_\alpha:=\bigl\{\xi_c:c\in P\bigr\}
\]
has cardinality $2^{\aleph_0}$. Together with \eqref{eq:incl}, this proves Theorem~\ref{thm:main}, with $C'$ in the role of $C$. \qed

\section{Remarks}

\begin{remark}\label{rem:uniform}
The set $C'$ depends on $\alpha$ only through the exponent $q$ of \eqref{eq:q}. Indeed, the sets $C$ and $C'$ are built from $R$, $\kappa$ and $q$ alone, and the estimates of Section~\ref{sec:proof} use the base $\beta=\alpha^q$ only through \eqref{eq:E}, which holds as soon as $\alpha^q>2q+1$. Thus, one and the same set $C'$ works for all $\alpha$ with $\alpha^q>2q+1$. In particular, for $\alpha>3$ we may take $q=1$, and then one and the same $C$ works for every $\alpha>3$.
\end{remark}

\begin{remark}\label{rem:range}
The only reason for passing to the power $\beta=\alpha^q$ is the shadowing error $E=1/(2(\beta-1))$ of Lemma~\ref{lem:shadow}. The drift $D=\kappa R/(R-1)$ can be made as small as we please by decreasing $\kappa$, but $E$ depends on the base alone, and the separation of $2^{\aleph_0}$ many $\xi$ at the end of Section~\ref{sec:proof} needs $2E<\tfrac12$. Indeed, two points of the circle differ by at most $\tfrac12$ in $\|\cdot\|$, so that \eqref{eq:close} alone does not exclude a common value of $\xi$ for two targets that differ by $\tfrac12$, unless $2E<\tfrac12$. For the base $\alpha$ itself, this means $\alpha>3$, and passing to $\alpha^q$ makes the error as small as we need. We have not tried to improve the constant in Lemma~\ref{lem:shadow}.
\end{remark}

\begin{remark}
The counterexamples lie densely in $(0,\infty)$. Since the index $n=0$ lies in no block, we may run the proof with $x_0=m$ and an arbitrary $y_0\in[0,1)$ in Lemma~\ref{lem:shadow}, instead of $x_0=1$ and $y_0=0$, and nothing changes. Given a real number $t\ge1$, the choice $m=\lfloor t\rfloor$ and $y_0=\{t\}$ gives $m+y_0=t$, and Lemma~\ref{lem:shadow} then yields $|\xi_c-t|\le E$ for every $c\in P$. Thus, with $E=1/(2(\alpha^q-1))$, the set $S$ of those $\xi>0$ for which $\spec((\xi\alpha^n)_{n\ge1})$ is uncountable meets every closed interval of length $2E$ in $[1,\infty)$. Moreover, $\spec((\alpha^{-k}\eta\,\alpha^n)_{n\ge1})=\spec((\eta\alpha^n)_{n\ge1})$ for every $\eta>0$ and every integer $k\ge0$, because $\alpha^{-k}\eta\,\alpha^n=\eta\alpha^{n-k}$, so that the two sequences $(\alpha^{-k}\eta\,\alpha^n-n\theta)_{n\ge1}$ and $(\eta\alpha^n-n\theta)_{n\ge1}$ differ by a shift of the index by $k$ and a translation by the constant $-k\theta$, neither of which affects uniform distribution. Hence, $\alpha^{-k}S=S$ for every $k\ge0$. Given $t>0$ and $\varepsilon>0$, we choose $k$ with $\alpha^kt\ge1$ and $\alpha^{-k}E<\varepsilon$, and we take $\xi\in S$ with $|\xi-\alpha^kt|\le E$. Then $\alpha^{-k}\xi\in S$ and $|\alpha^{-k}\xi-t|\le\alpha^{-k}E<\varepsilon$. Thus, $S$ is dense in $(0,\infty)$, although it is a Lebesgue null set.
\end{remark}

\begin{remark}
The set $C$ has Hausdorff dimension zero, and so has $C'=C/q$. Indeed, along a branch we have $n_{k+1}\ge2n_k$ and $n_0=1$, and so $n_K\ge2^K$. Thus, the $K$-th level of the construction covers $C$ by $2^K$ intervals of length at most $2\kappa R^{-2^K}$, and for every $d>0$ the sum of the $d$-th powers of these lengths is at most $2^K\bigl(2\kappa R^{-2^K}\bigr)^d$, which tends to $0$ as $K\to\infty$. This is no accident. Mend\`es France states in \cite[Probl\`eme~5]{MF}, with reference to Wallin \cite{Wa}, that the spectrum of every real sequence has Hausdorff dimension zero. Thus, a spectrum is always small in dimension, and Theorem~\ref{thm:main} shows that it can nevertheless be uncountable.
\end{remark}

\section*{Acknowledgements}

I thank Faruk Temur for drawing my attention to this problem, for encouraging me to work on it, and for reading the proofs and making valuable comments.

\end{document}